\documentclass[a4paper, 12pt]{amsart}

\usepackage{mathptmx, amsmath, tabularx, amsthm, amssymb, amsfonts, amscd}
\usepackage{eulervm}
\usepackage{paralist}
\usepackage{aliascnt}
\usepackage{blkarray}
\usepackage{mathbbol}
\usepackage[inner=2.3cm,outer=2.3cm, bottom=3.3cm]{geometry}
\usepackage{setspace}
\usepackage[colorlinks=true,linkcolor=blue,urlcolor=blue,pagebackref]{hyperref}
\usepackage[initials]{amsrefs}
\usepackage{marginnote}
\usepackage{tikz}
\usetikzlibrary{matrix}
\usetikzlibrary{arrows,calc}
\allowdisplaybreaks

\BibSpec{collection.article}{%
	+{}  {\PrintAuthors}                {author}
	+{,} { \textit}                     {title}
	+{.} { }                            {part}
	+{:} { \textit}                     {subtitle}
	+{,} { \PrintContributions}         {contribution}
	+{,} { \PrintConference}            {conference}
	+{}  {\PrintBook}                   {book}
	+{,} { }                            {booktitle}
	+{,} { }                            {series}
	+{, vol.} { }                            {volume}
	+{,} { }                            {publisher}
	+{,} { \PrintDateB}                 {date}
	+{,} { pp.~}                        {pages}
	+{,} { }                            {status}
	+{,} { \PrintDOI}                   {doi}
	+{,} { available at \eprint}        {eprint}
	+{}  { \parenthesize}               {language}
	+{}  { \PrintTranslation}           {translation}
	+{;} { \PrintReprint}               {reprint}
	+{.} { }                            {note}
	+{.} {}                             {transition}
	+{}  {\SentenceSpace \PrintReviews} {review}
}

\newcommand{\RR}{\mathbb{R}}

\newcommand{\NN}{\normalfont\mathbb{N}}
\newcommand{\ZZ}{{\normalfont\mathbb{Z}}}

\newcommand{\PP}{{\normalfont\mathbb{P}}}

\newcommand{\dd}{{\normalfont\mathbf{d}}}

\newcommand{\mm}{{\normalfont\mathfrak{m}}}

\newcommand{\bn}{{\normalfont\mathbf{n}}}

\newcommand{\bm}{{\normalfont\mathbf{m}}}
\newcommand{\ttt}{{\normalfont\mathbf{t}}}

\newcommand{\rank}{\normalfont\text{rank}}

\newcommand{\sat}{{\normalfont\text{sat}}}

\newcommand{\Quot}{\normalfont\text{Quot}}

\newcommand{\II}{\mathbb{I}}

\newcommand{\JJ}{\mathbb{J}}

\newcommand{\bJ}{\mathbf{J}}
\newcommand{\bI}{\mathbf{I}}

\newcommand{\R}{\mathcal{R}}

\newcommand{\LL}{\mathbb{L}}

\newcommand{\Spec}{\normalfont\text{Spec}}

\newcommand{\Vol}{{\normalfont\text{Vol}}}

\newcommand{\ind}{{\normalfont\text{ind}}}
\newcommand{\Con}{{\normalfont\text{Con}}}

\def\fb{\mathbf{b}}
\def\f0{\mathbf{0}}

\def\fb{\mathbf{b}}

\def\fp{\mathfrak{p}}

\def\fn{\mathbf{n}}
\def\fm{\mathbf{m}}

\def\ft{\mathbf{t}}
\def\fd{\mathbf{d}}
\def\fy{\mathbf{y}}
\def\ft{\mathbf{t}}
\def\bt{\mathbf{t}}

\def\ls{\leqslant}
\def\gs{\geqslant}

\def\*{{\color{red}\blacksquare}}

\newtheorem{theorem}{Theorem}[section]

\newtheorem{headthm}{Theorem}

\newaliascnt{headcor}{headthm}

\aliascntresetthe{headcor}

\newaliascnt{headthmdef}{headthm}

\aliascntresetthe{headthmdef}

\newaliascnt{headconj}{headthm}

\aliascntresetthe{headconj}

\newaliascnt{corollary}{theorem}
\newtheorem{corollary}[corollary]{Corollary}
\aliascntresetthe{corollary}

\newaliascnt{lemma}{theorem}
\newtheorem{lemma}[lemma]{Lemma}
\aliascntresetthe{lemma}

\newaliascnt{conjecture}{theorem}

\aliascntresetthe{conjecture}

\newaliascnt{proposition}{theorem}
\newtheorem{proposition}[proposition]{Proposition}
\aliascntresetthe{proposition}

\theoremstyle{definition}
\newaliascnt{definition}{theorem}
\newtheorem{definition}[definition]{Definition}
\aliascntresetthe{definition}

\newaliascnt{notation}{theorem}

\aliascntresetthe{notation}

\newaliascnt{example}{theorem}
\newtheorem{example}[example]{Example}
\aliascntresetthe{example}

\newaliascnt{examples}{theorem}

\aliascntresetthe{examples}

\newaliascnt{remark}{theorem}
\newtheorem{remark}[remark]{Remark}
\aliascntresetthe{remark}

\newaliascnt{problem}{theorem}

\aliascntresetthe{problem}

\newaliascnt{question}{theorem}

\aliascntresetthe{question}

\newaliascnt{convention}{theorem}

\aliascntresetthe{convention}

\newaliascnt{construction}{theorem}

\aliascntresetthe{construction}

\newaliascnt{setup}{theorem}
\newtheorem{setup}[setup]{Setup}
\aliascntresetthe{setup}

\newaliascnt{algorithm}{theorem}

\aliascntresetthe{algorithm}

\newaliascnt{observation}{theorem}

\aliascntresetthe{observation}

\newaliascnt{defprop}{theorem}

\aliascntresetthe{defprop}

\def\equationautorefname~#1\null{(#1)\null}
\def\sectionautorefname~#1\null{Section #1\null}
\def\subsectionautorefname~#1\null{\S #1\null}

\begin{document}

\title[Mixed multiplicities of graded families of ideals]{Mixed multiplicities of graded families of ideals}

\author[Yairon Cid-Ruiz]{Yairon Cid-Ruiz}
\address[Cid-Ruiz]{Department of Mathematics: Algebra and Geometry, Ghent University, Krijgslaan 281 – S25, 9000 Ghent, Belgium}
\email{Yairon.CidRuiz@ugent.be}
\urladdr{https://ycid.github.io}

\author[Jonathan Monta\~no]{Jonathan Monta\~no$^{*}$}
\address[Monta\~no ]{Department of Mathematical Sciences  \\ New Mexico State University  \\PO Box 30001\\Las Cruces, NM 88003-8001}
\thanks{$^{*}$ The second  author is  supported by NSF Grant DMS \#2001645.}
\email{jmon@nmsu.edu}

\date{\today}
\keywords{}
\subjclass[2010]{Primary 13H15.}

\begin{abstract}
We show the existence (and define) the mixed multiplicities of arbitrary graded families of ideals  under mild assumptions.
In particular, our methods and results are valid for the case of arbitrary $\mm$-primary graded families.
Furthermore, we provide a far-reaching ``Volume = Multiplicity formula'' for the mixed multiplicities of graded families of ideals.
\end{abstract}
\maketitle

\vspace*{-.5cm}

\section{Introduction}

The concept of mixed multiplicities of ideals is of remarkable importance in the areas of commutative algebra and algebraic geometry, and its study seems to have been initiated by Bhattacharya in \cite{Bhattacharya}.
These multiplicities have a successful history of interconnecting problems from commutative algebra,  with applications to the topics of Milnor numbers, mixed volumes, and integral dependence (see, e.g., \cite{Huh12,huneke2006integral,TRUNG_VERMA_MIXED_VOL,teissier1973cycles,Bhattacharya,TRUNG_POSITIVE}).  
For comprehensive discussions on them, the reader is referred to the survey \cite{TRUNG_VERMA_SURVEY} and to Chapter $17$ of the book \cite{huneke2006integral}.

\medskip

This concept can be naturally extended to graded families of ideals -- it amounts to consider graded families of ideals instead of just the powers of ideals.
A \emph{graded family} of ideals $\II=\{I_n\}_{n \in \NN}$ in a ring $R$ is a sequence of ideals such that $I_0 = R$ and $I_nI_m  \subset I_{n+m}$ for every $n, m \in \NN$.
 If in addition the Rees algebra $\R(\II)=\oplus_{n\in \NN}I_nt^n \subset R[t]$ is Noetherian, then we say $\II$ is {\it Noetherian}.
When $I_n \supseteq I_{n+1}$ for every $n \in \NN$, we say that $\II$ is a \emph{filtration}.
If $(R,\mm)$ is local with maximal ideal $\mm$, we say that $\II$ is \emph{$\mm$-primary} when $I_n$ is $\mm$-primary for each $n \ge 1$.

\medskip

The study of mixed multiplicities of (not necessarily Noetherian) graded families  was  pioneered by Cutkosky-Sarkar-Srinivasan \cite{cutkosky2019} for the case of \emph{$\mm$-primary filtrations}.
Recently, in the previous work \cite{MIXED_VOL_MONOM}, the authors of this paper defined mixed multiplicities for \emph{arbitrary graded families of monomial ideals}  (that satisfy the mild condition of having a linear bound for the degree of the generators of the ideals), and showed that the mixed volumes of arbitrary convex bodies can be expressed in terms of the newly defined mixed multiplicities.
The latter result provides an important application of mixed multiplicities of graded families and gives some reinforcement on the interest of studying this notion.

\medskip

The goal of this paper is to show the existence (and define) the mixed multiplicities of arbitrary graded families of ideals  under mild assumptions.
In the $\mm$-primary case, the conditions that we assume are automatically satisfied, and so we obtain an extension of the main result in \cite{cutkosky2019} (i.e., we drop the filtration condition).
An additional important result of our work is that we show a ``Volume = Multiplicity formula'' for mixed multiplicities of graded families.
Below, we discuss the main contributions of this paper.

\medskip

\noindent
\emph{1.1 -- Mixed multiplicities of $\mm$-primary graded families of ideals.}
Let $(R,\mm,k)$ be a Noetherian local ring of dimension $d$. Let $M$ be a finitely generated $R$-module and let $I_1,\ldots, I_s$ be $\mm$-primary ideals.  
Then, for $\bm=(m_1,\ldots,m_s)\gg \mathbf{0}$ the function $\lambda(M/I_1^{m_1}\cdots I_s^{m_s}M)$ coincides with a polynomial in $m_1,\ldots, m_s$ of total degree equal to $\dim(M)$. 
The  homogeneous part  in degree $d=\dim(R)$ of this polynomial can be written as 
\begin{equation}\label{polMMnot}
G_{(I_1,\ldots, I_s)}^M(t_1,\ldots, t_s):=\sum_{ \fd=(d_1,\ldots, d_s)\in \NN^{s},\,|\fd|=d}\frac{ e_{\fd}(M; I_1,\ldots, I_s)}{d_1!\cdots d_s!} t_1^{d_1}\cdots t_s^{d_s}.
\end{equation}
The numbers $e_{\fd}(M; I_1,\ldots, I_s)$ are non-negative integers called the {\it mixed multiplicities} of $M$ with respect to $I_1,\ldots, I_s$.

Motivated by \autoref{polMMnot}, the existence of a similar polynomial for the case of graded families yields the definition of mixed multiplicities.
Let  $\II(1)=\{I(1)_n\}_{n\in \NN}$,   $\ldots$, $\II(s)=\{I(s)_n\}_{n\in \NN}$  be (not necessarily Noetherian) $\mm$-primary graded families of   ideals in $R$.
The following theorem extends the main result of \cite{cutkosky2019} from $\mm$-primary filtrations to $\mm$-primary graded families.

\begin{headthm}[\autoref{thm_exists_m-prim}, \autoref{cor_vol=mult}]
	\label{thmA}
	Adopt the assumptions and notations above, and suppose that $\dim (N(\hat{R}))<d$,  where $N(\hat{R})$ denotes the nilradical of the $\mm$-adic completion  $\hat{R}$.
	Then, there exists a homogeneous polynomial of total degree $d$ and non-negative real coefficients, denoted by $G_{\left(\II(1),\ldots,\II(s)\right)}^M(t_1,\ldots, t_s)\in \RR[t_1,\ldots, t_s]$, such that 
	$$
	G_{\left(\II(1),\ldots,\II(s)\right)}^M(m_1,\ldots, m_s) \;=\; \lim_{m\to \infty}\frac{\lambda\left( M/I(1)_{mm_1}\cdots I(s)_{mm_s}M\right)}{m^d},
	$$
	for every $(m_1,\ldots, m_s) \in \NN^s$.
	Moreover, for every $(m_1,\ldots, m_s) \in \NN^s$ we have 
	$$
	\lim_{p\to\infty} \frac{G_{\left(I(1)_p,\ldots,I(s)_p\right)}^M(m_1,\ldots, m_s)}{p^d} \;=\, G_{\left(\II(1),\ldots,\II(s)\right)}^M(m_1,\ldots, m_s).
	$$
\end{headthm}

We can write the polynomial  $G_{\left(\II(1),\ldots,\II(s)\right)}^M(t_1,\ldots, t_s)$ from \autoref{thmA}  as follows
$$
G_{\left(\II(1),\ldots,\II(s)\right)}^M(t_1,\ldots, t_s) \;=\; \sum_{ \fd=(d_1,\ldots, d_s)\in \NN^{s},\,|\fd|=d}
 \frac{e_{\fd}(M; \II(1),\ldots, \II(s))}{d_1!\cdots d_s!} \, t_1^{d_1}\cdots t_s^{d_s}.
$$
Then, for each $\fd \in \NN^{s}$ with $|\dd| = d$, one defines the real number $e_{\fd}(M; \II(1),\ldots, \II(s)) \ge 0$ to be the {\it mixed multiplicity} of  $M$ with respect to $\II(1), \ldots, \II(s)$ of type $\dd$ (see \autoref{defMMm-prim}).
An important consequence of \autoref{thmA} is a ``Volume = Multiplicity formula'' for mixed multiplicities, that is, we obtain the following equality 
$$
\lim_{p\to \infty} \frac{e_{\fd}(M; I(1)_p,\ldots, I(s)_p)}{p^d} \;=\; e_{\fd}(M; \II(1),\ldots, \II(s))
$$
for each $\fd \in \NN^{s}$ with $|\dd| = d$ (see \autoref{cor_vol=mult}).
The latter result extends the ``Volume = Multiplicity formula'' known for the case of multiplicities (see \cite{ein2003,must02,lazarsfeld09,cutkosky2013,cutkosky2014,cutkosky2015}) and it serves as the main tool to provide simple proofs of important properties of these mixed multiplicities; the list includes: additivity under short exact sequences (see \autoref{additivity}), Associativity formula (see \autoref{associativity}), and Minkowski inequalities (see \autoref{Minkowski}). 

\medskip
\noindent
\emph{ 1.2 -- Mixed multiplicities of arbitrary graded families of ideals.}
Assume now that $R$  has positive dimension. Let $I, J_1, \ldots, J_r$ be ideals in $R$ such that  $I$ is $\mm$-primary and $J_1, \ldots, J_r$ have positive grade, i.e., $J_i$ contains non-zero divisors for each $1\ls i\ls r$.  
Then,  for $n_0\gg 0$ and $\bn=(n_1,\ldots,n_r)\gg \mathbf{0}$ the function $\lambda(I^{n_0}J_1^{n_1}\cdots J_r^{n_r}/I^{n_0+1}J_1^{n_1}\cdots J_r^{n_r})$ coincides with  a polynomial  of total degree $d-1$  whose homogeneous part  in degree $d-1$ can be written as 
$$\sum_{(d_0,\fd)=(d_0,d_1,\ldots, d_r)\in \NN^{r},\, d_0+|\fd|=d-1}\frac{ e_{(d_0,\fd)}(I\mid J_1,\ldots, J_r)}{d_0!d_1!\cdots d_r!} t_0^{d_0}t_1^{d_1}\cdots t_r^{d_r}.$$
Following standard techniques (see, e.g., \cite[proof of Lemma 4.2]{MIXED_VOL_MONOM}), one may show that for each $n_0\in \NN$ and $\bn=(n_1,\ldots, n_r)\in \NN^r$ the limit $\lim_{m\rightarrow \infty}\frac{\lambda \left( J_1^{mn_1}\cdots J_r^{mn_r}/I^{mn_0}J_1^{mn_1}\cdots J_r^{mn_r} \right)}{m^{d}}$ exists and coincides with $G_{(I;J_1,\ldots, J_r)}(n_0,n_1,\ldots, n_r)$, where $G_{(I;J_1,\ldots, J_r)}(t_0,t_1,\ldots, t_r)$ is the following  polynomial 
\begin{equation}\label{polGenMMnot}
G_{(I;J_1,\ldots, J_r)}(t_0,t_1,\ldots, t_r):=\sum_{(d_0,\dd) \in \NN^{r+1}, d_0+|\fd|=d-1}\frac{ e_{(d_0,\fd)}(I\mid J_1,\ldots, J_r)}{(d_0+1)!d_1!\cdots d_r!}\, t_0^{d_0+1}t_1^{d_1}\cdots t_r^{d_r}.
\end{equation}
The numbers $e_{(d_0,\fd)}(I\mid J_1,\ldots, J_r)$ are non-negative integers called the {\it mixed multiplicities} of $J_1,\ldots, J_r$ with respect to $I$.

Now, the notion of mixed multiplicities for non $\mm$-primary graded families can be obtained by showing the existence of a polynomial similar to the one in \autoref{polGenMMnot}.
Let $\JJ(1)=\{J(1)_n\}_{n\in \NN}$,   $\ldots$, $\JJ(r)=\{J(r)_n\}_{n\in \NN}$ be (not necessarily Noetherian) graded families of non-zero ideals, and let $\II=\{I_n\}_{n\in \NN}$ be a (not necessarily Noetherian) $\mm$-primary graded family of  ideals.  Moreover, assume that for every $n_0\in \NN$ and $\bn=(n_1,\ldots, n_r)\in \NN^{r}$ the pair of graded families 
  $$
  \Big(\{\bJ_{m\bn}\}_{m\in \NN},\, \{I_{mn_0}\bJ_{m\bn}\}_{m\in \NN}\Big)
  $$ 
  satisfies a certain  linear growth condition (see \autoref{linGr}, \autoref{setup_gen_ex}, \autoref{rem_lin_G}, and \autoref{remLinGr} for details).

\begin{headthm}[\autoref{thm_exists_general}, \autoref{cor_vol=mult_general}]
	\label{thmB}
	Adopt the assumptions and notations above, and suppose that $R$ is analytically irreducible.
	Then, there exists a homogeneous polynomial of total degree $d$ and non-negative real coefficients $G_{\left(\II;\JJ(1),\ldots,\JJ(r)\right)}(t_0,t_1,\ldots, t_r) \in \RR[t_0,t_1,\ldots, t_r]$ such that 
	$$
	G_{\left(\II;\JJ(1),\ldots,\JJ(r)\right)}(n_0,n_1,\ldots, n_r) \;=\; \lim_{m\to \infty}\frac{\lambda\left( J(1)_{mn_1}\cdots J(r)_{mn_r}/I_{mn_0}J(1)_{mn_1}\cdots J(r)_{mn_r}\right)}{m^d},
	$$
	for every $(n_0,n_1,\ldots, n_r) \in \NN^{r+1}$.
	Additionally, the polynomial $G_{\left(\II;\JJ(1),\ldots,\JJ(r)\right)}(t_0,t_1,\ldots, t_r)$ has no term of the form $et_1^{d_1}\cdots t_r^{d_r}$ with $e\neq 0$ and $d_1+\cdots+d_r=d$.
	Moreover, for every $(n_0,n_1,\ldots, n_r) \in \NN^{r+1}$ we have 
	$$
	\lim_{p\to\infty} \frac{G_{\left(I_p;J(1)_p,\ldots,J(r)_p\right)}(n_0,n_1,\ldots, n_r)}{p^d} \;=\; G_{\left(\II;\JJ(1),\ldots,\JJ(r)\right)}(n_0,n_1,\ldots, n_r).
	$$
\end{headthm}

By \autoref{thmB}  we can write $G_{\left(\II;\JJ(1),\ldots,\JJ(r)\right)}(t_0,t_1,\ldots, t_r)$ as 
$$
G_{\left(\II;\JJ(1),\ldots,\JJ(r)\right)}(t_0,t_1,\ldots, t_r) \;=\; \sum_{(d_0,\dd) \in \NN^{r+1}, d_0+|\fd|=d-1}\frac{ e_{(d_0,\fd)}(\II\mid \JJ(1),\ldots, \JJ(r))}{(d_0+1)!d_1!\cdots d_r!}\, t_0^{d_0+1}t_1^{d_1}\cdots t_r^{d_r}.
$$
Then, for each $(d_0,\fd) \in \NN^{r+1}$ with $d_0+|\dd| = d-1$, one defines the real number $e_{(d_0,\fd)}(\II \mid \JJ(1),\ldots, \JJ(r)) \ge 0$ to be the {\it mixed multiplicity} of  $\JJ(1), \ldots, \JJ(r)$ with respect to $\II$ of type $(d_0,\dd)$ (see \autoref{defMMm-general}).
Again, \autoref{thmB} yields a ``Volume = Multiplicity formula'' for this case, that is, we obtain the following equality 
$$
\lim_{p\to \infty} \frac{e_{(d_0,\fd)}(I_p\mid J(1)_p,\ldots, J(r)_p)}{p^d} \;=\; e_{(d_0,\fd)}(\II\mid \JJ(1),\ldots, \JJ(r))
$$
for each $(d_0,\fd) \in \NN^{r+1}$ with $d_0+|\dd| = d-1$ (see \autoref{cor_vol=mult_general}).

\medskip
\noindent
\emph{1.3 -- Some notations and organization of the paper.} 
For a vector $\fn=(n_1,\ldots, n_r)\in \NN^r$ we denote by $|\fn|$ the sum of its entries. 
For vectors  $\fn=(n_1\ldots, n_r)$ and $\fm=(m_1,\ldots, m_r)$ in $\NN^r$ we write $\fn\gs \fm$ if $n_i\gs m_i$ for every $1\ls i\ls r$,  we write $\fn\gg \mathbf{0}$ if $n_i\gg 0$ for every $1\ls i\ls r$. 
The vectors $(0,\ldots, 0)\in \NN^r$ and $(1,\ldots, 1)\in \NN^r$ are denoted by $\mathbf{0}$ and $\mathbf{1}$, respectively. 

The basic outline of this paper is as follows. 
\autoref{sec_linG} is of technical nature and there we deal with certain limits which are the core of our arguments.
In \autoref{sect_exists} we prove \autoref{thmA} and \autoref{thmB}.
Finally, \autoref{sect_properties} is devoted to showing some properties of mixed multiplicities of $\mm$-primary graded families of ideals.

\section{Linear growth }\label{sec_linG}

This technical section contains the core of our methods. 
In \autoref{thm_main} below we show the equality of certain limits and that result allows us to define mixed multiplicities of graded families in the next section.
Throughout this section the following setup is fixed.

\begin{setup}\label{setup_lin_gr}
Let $(R,\mm,k)$ be a $d$-dimensional complete local domain. 
Let $\JJ=\{J_n\}_{n\in \NN}$ and $\II=\{I_n\}_{n\in \NN}$ be (not necessarily Noetherian) graded families of  non-zero ideals, such that $J_n\supseteq I_n$ for every $n\in \NN$.  
For every  $a\in \ZZ_{>0}$, let $\JJ_a :=\{J_{a, n}\}_{n\in \NN}$ be the Noetherian graded family generated by $J_1,\ldots, J_a$, that is, for $n>a$ one has $J_{a,n}=\sum_{i=1}^{n-1} J_{a,i}J_{a,n-i}$.  
Likewise, define $\II_a :=\{I_{a,n}\}_{n\in\NN}$.  
\end{setup}

The following definition includes a condition on graded families that is needed in the proof of our main result. 

\begin{definition}\label{linGr}
Assume \autoref{setup_lin_gr}. We say that the pair of graded families   $(\JJ,\II)$ has {\it linear growth} if   there exists $c= c(\JJ,\II) \in \NN$ such that 
\begin{align*}
	& J_{n}\cap \mm^{cn}=I_n\cap \mm^{cn} \mbox{ for every }n\in \NN.
\end{align*}
\end{definition}

\begin{remark}\label{rem_lin_G}
	We note that the above condition is quite natural and that it holds in a variety of interesting cases. For example, if $\II$ and $\JJ$  are $\mm$-primary, this condition is automatically satisfied (see \autoref{theMpLG} and \autoref{remLinGr}). Moreover, if $\II$ is Noetherian, then $(\JJ,\II)$ has linear growth as long as $\lambda(J_n/I_n)<\infty$ for every $n\in \NN$ (see \autoref{Noeth_lin_gr}).
\end{remark}

For an ideal $I$, we denote by $I^{\sat}:= (I:_R\mm^{\infty})$ the  saturation of $I$. We also write $\II^{\sat}=\{I_n^{\sat}\}_{n\in \NN}$ and call it the {\it saturation} of $\II$. 
The following states that the pair of a Noetherian filtration and its saturation has linear growth.

\begin{proposition}[{\cite[Theorem 3.4]{swanson97}}]\label{Noeth_lin_gr}
Assume \autoref{setup_lin_gr} and that $\II$ is Noetherian. Then $(\II^{\rm sat}, \II)$ has linear growth. 
\end{proposition}
\begin{proof}
Set $I_n=R$ for $n<0$. Since the algebra $R[\II t, t^{-1}]=\bigoplus_{n\in \ZZ}I_nt^n\subset R[t,t^{-1}]$ is Noetherian, the same proof of \cite[Theorem 3.4]{swanson97} applies to show that there exists $c\in \NN$ such that,  for all $n\gs 1$,  there exists a primary decomposition of $I_n$ whose $\mm$-primary component (if any) contains $\mm^{cn}$.  Therefore, $I_n^{\rm sat}\cap \mm^{cn}\subseteq I_n\cap \mm^{cn}$, and  the result follows.
\end{proof}

We now recall some notation from \cite{KAVEH_KHOVANSKII} (see also \cite{cutkosky2014}).  Let $S \subset \NN^{d+1}$ be a subsemigroup of $\NN^{d+1}$ and $\pi  : \RR^{d+1} \rightarrow \RR$ the projection onto the last coordinate. 
For any $n\in \NN$ we define
$$
S_n \;:=\; S\cap \pi^{-1}(n), \qquad \text{ and } \qquad
n\star U \;:=\; \Big\{\sum_{i=1}^n \bn_i \,\mid \,\bn_1,\ldots, \bn_n \in U \Big\}
$$ 
for any subset $U \subset S$ of $S$.
Let $L = L(S)$ be the subspace of $\RR^{d+1}$  generated by $S$ and  $M = M(S)$ the rational half-space $M(S):= L(S) \cap \pi^{-1}\left(\RR_{\gs 0}\right)$. 
Let $\Con(S) \subset L(S)$ be the closed convex cone given as the closure of the set of all linear combinations $\sum_i \lambda_i\bn_i$ with $\bn_i \in S$ and $\lambda_i \ge 0$.
Let $G(S) \subset L(S)$ be the group generated by $S$. The pair $(S,M)$ is  \emph{strongly admissible} if $S \subset M$, $\Con(S)$ is strictly convex, and  $\Con(S)\cap \partial M=\{\mathbf{0}\}$. 
For such a pair we define 
$$
 \ind(S,M) \;:=\; \left[\ZZ : \pi(G(S))\right]\qquad \text{and}\qquad \ind(S,\partial M)\; :=\; \left[\partial M_\ZZ : G(S) \cap \partial M\right].
 $$ 
Moreover,  $$\Delta(S, M) \;:=\; \Con(S) \cap \pi^{-1}(\ind(S,M) )$$  is the {\it Newton-Okounkov body} of $(S,M)$ and, if $q=\dim(\partial M)$,  $\Vol_q(\Delta(S,M))$ is its  \emph{integral volume}.

Let $(\JJ, \II)$ be a pair of graded families with linear growth and let $c=c(\JJ,\II)$ be as in \autoref{linGr}. In \cite{cutkosky2014} (see also \cite[Lemma 4.2]{cutkosky2019}) Cutkosky showed the existence of an excellent regular local ring $(S,\mm_S, k_S)$ of dimension $d=\dim(R)$ that birrationally dominates $R$, that is, $S$ is essentially of finite type over $R$ and the two rings have the same quotient field $\Quot(R)$. Given $\fy= y_1,\ldots, y_d$ a generating set for $\mm_S$, and $\mathbf{b}=(b_1,\ldots, b_d)$ rationally independent  real numbers with $b_i\gs 1$ for every $1\ls i\ls d$,  he constructed a valuation $v$ on $\Quot(R)$ by setting $v(\fy^\fn)=\fn\cdot \mathbf{b} = n_1b_1 + \cdots + n_db_d$ for every $\fn = (n_1,\ldots,n_d)\in \NN^d$. 
Moreover, it is shown that $v$ dominates $S$, that is, $\mm_v\cap S=\mm_S$, where $(V_v, \mm_v, k_v)$ is the valuation ring of $v$; and that $k_v=k_S$. 

For every $0 \le \beta \in \RR$ we define the following ideals of $V_v$:
$$
K_\beta :=\{f\in \Quot(R)\mid v(f)\gs \beta\}, \quad \text{ and } \quad K_\beta^+ :=\{f\in \Quot(R)\mid v(f)> \beta\}.
$$
We note that  $K_\beta/K_\beta^+\cong k_v=k_S$ for every $\beta \in \Gamma_\nu \subset \RR$ inside the valued group $\Gamma_\nu$, thus for any $R$-ideal $I$ one has 
\begin{equation}
	\label{maxDim}\dim_k( K_\beta \cap I/ K_\beta^+ \cap I) \,\ls\, \dim_k( K_\beta / K_\beta^+ ) \,\ls\, [k_S:k].
\end{equation}
By  \cite[Lemma 4.3]{cutkosky2013},  there exists $\alpha\in \ZZ_{>0}$ such that $K_{\alpha n}\cap R\subseteq \mm^n$ for every $n\in \NN$. 
Therefore, the condition assumed in \autoref{linGr} yields that
$$
K_{\alpha c n}\cap J_n =K_{\alpha cn}\cap I_n, \mbox{ for every }n\in \NN.
$$

Thus, since $K_{\alpha c n}\cap R$ is an $\mm$-primary ideal, for every $n\in \NN$ one has
\begin{equation}\label{splitLength}
\lambda(J_n/I_n)=\lambda(J_n/K_{\alpha cn}\cap J_n)-\lambda(I_n/K_{\alpha cn}\cap I_n).
\end{equation}

For every $t\in \ZZ_{>0}$  we define:
\begin{align}
	\label{smgps}
	\begin{split}
\Gamma_{\JJ}^{(t)} &:=\big\{(\fn,n)=(n_1,\ldots, n_d,n)\in \NN^{d+1}\mid \dim_k\left(K_{\fn\cdot \fb}\cap J_n/K^+_{\fn\cdot \fb}\cap J_n \right)\gs t\mbox{ and }|\fn|\ls \alpha c n\big\},\\
 \Gamma_{\II}^{(t)} &:=\big\{(\fn,n)=(n_1,\ldots, n_d,n)\in \NN^{d+1}\mid \dim_k\left(K_{\fn\cdot \fb}\cap I_n/K^+_{\fn\cdot \fb}\cap I_n \right)\gs t\mbox{ and }|\fn|\ls \alpha c n\big\};
	\end{split}
\end{align}
and for every $a,t\in \ZZ_{>0}$  we define:
\begin{align}\label{smgpsA}
	\begin{split}
\Gamma_{\JJ_a}^{(t)} &:= \{(\fn,n)=(n_1,\ldots, n_d,n)\in \NN^{d+1}\mid \dim_k\left(K_{\fn\cdot \fb}\cap J_{a,n}/K^+_{\fn\cdot \fb}\cap J_{a,n} \right)\gs t\mbox{ and }|\fn|\ls \alpha c n \},\\
 \Gamma_{\II_a}^{(t)} &:= \big\{(\fn,n)=(n_1,\ldots, n_d,n)\in \NN^{d+1}\mid \dim_k\left(K_{\fn\cdot \fb}\cap I_{a,n}/K^+_{\fn\cdot \fb}\cap I_{a,n} \right)\gs t\mbox{ and }|\fn|\ls \alpha c n \big\}.
 	\end{split}
\end{align}

Let $S$ be any of the sets defined in either \autoref{smgps} or \autoref{smgpsA}. 
As noted in \cite[Theorem 6.1]{cutkosky2014}, one can adapt the proofs of \cite[Lemma 4.4 and Lemma 4.5]{cutkosky2014} to show $S$ is a semigroup. Moreover, one has that  $$G(S)=\ZZ^{d+1}$$
and if $M=\RR^d\times \RR_{\gs 0} = \pi^{-1}(\RR_{\ge 0})$, then 
$(S, M)$ is strongly admissible with 
\begin{equation}\label{ind1}
\ind(S, M)=\ind(S, \partial M)=1.
\end{equation}
Set $\Delta(S):=\Delta(S, M)$. 

The following lemma is of fundamental importance for our main results (cf. \cite[Theorem 6.1]{cutkosky2014}).

\begin{lemma}\label{thm_limits}
Assume  the notations introduced in this section, in particular that $(\JJ, \II)$ has linear growth, then the following  limit exists 
$$
\displaystyle\lim_{n\to \infty}\frac{\lambda(J_n/I_n)}{n^d},
$$
and is equal to 
$$\displaystyle\sum_{t=1}^{[k_S:k]}\left(\Vol_d\left(\Delta\big(\Gamma_{\JJ}^{(t)}\big)\right)-\Vol_d\left(\Delta\big(\Gamma_{\II}^{(t)}\big)\right)\right).$$
Additionally, we have the equalities
$$
\lim_{a\to \infty}\Vol_d\left(\Delta\big(\Gamma_{\JJ_a}^{(t)}\big)\right) = \Vol_d\left(\Delta\big(\Gamma_{\JJ}^{(t)}\big)\right)
\quad \text{ and } \quad
\lim_{a\to \infty}\Vol_d\left(\Delta\big(\Gamma_{\II_a}^{(t)}\big)\right) = \Vol_d\left(\Delta\big(\Gamma_{\II}^{(t)}\big)\right)
$$
for all $t \in \ZZ_{>0}$.
\end{lemma} 
\begin{proof}
By \autoref{splitLength} , and since $b_i\gs 1$ for every $1\ls i\ls d$, 
 we have the following equalities
\begin{align*}
\lambda(J_n/I_n)&=\sum_{0\ls \fn\cdot \fb\ls \alpha  c n} \Big(\dim_k\left(K_{\fn\cdot \fb}\cap J_n/K^+_{\fn\cdot \fb}\cap J_n \right)-\dim_k\left(K_{\fn\cdot \fb}\cap I_n/K^+_{\fn\cdot \fb}\cap I_n \right)\Big)\\
&=\sum_{0\ls |\fn| \ls \alpha c n}\Big(\dim_k\left(K_{\fn\cdot \fb}\cap J_n/K^+_{\fn\cdot \fb}\cap J_n \right)-\dim_k\left(K_{\fn\cdot \fb}\cap I_n/K^+_{\fn\cdot \fb}\cap I_n \right)\Big)\\
&=\sum_{t=1}^{[k_S:k]} \left( \# \Gamma_{\JJ,n}^{(t)} \,-\, \#\Gamma_{\II,n}^{(t)}\right),
\end{align*}
where the last equality holds by \autoref{maxDim}. Thus, from \cite[Corollary 1.16]{KAVEH_KHOVANSKII} and \autoref{ind1} we obtain 
\begin{equation}\label{theLL}
\displaystyle\lim_{n\to \infty}\frac{\lambda(J_n/I_n)}{n^d}=\displaystyle\sum_{t=1}^{[k_S:k]}\left(\Vol_d\left(\Delta\big(\Gamma_{\JJ}^{(t)}\big)\right)-\Vol_d\left(\Delta\big(\Gamma_{\II}^{(t)}\big)\right)\right),
\end{equation} 
and so the first statement follows. 

Now, for every  $n,a,t\in \ZZ_{>0}$ we have
$n\star \Gamma_{\JJ,a}^{(t)}\subseteq \Gamma_{\JJ_{a,na}}^{(t)}$ and $n\star \Gamma_{\II,a}^{(t)}\subseteq \Gamma_{\II_{a,na}}^{(t)}.$ 
Thus,  by \cite[Proposition 3.1]{lazarsfeld09} (see also \cite[Theorem 3.3]{cutkosky2014}) and \cite[Corollary 1.16]{KAVEH_KHOVANSKII}, for a fixed $\varepsilon\in \RR_{>0}$,  there exists $a_0\in \NN$ such that if $a\gs a_0$ we have 
\begin{equation}\label{cotaJ}
\Vol_d\left(\Delta\big(\Gamma_{\JJ}^{(t)}\big)\right)\gs \Vol_d\left(\Delta\big(\Gamma_{\JJ_a}^{(t)}\big)\right)=\displaystyle\lim_{n\to \infty}\frac{\#\left[ \Gamma_{\JJ_{a,na}}^{(t)}\right] }{(na)^d}\gs\displaystyle\lim_{n\to \infty}\frac{\#\left [n\star \Gamma_{\JJ,{a}}^{(t)} \right]}{(na)^d}\gs \Vol_d\left(\Delta\big(\Gamma_{\JJ}^{(t)}\big)\right)-\varepsilon,
\end{equation}
\begin{equation}\label{cotaI}
\Vol_d\left(\Delta\big(\Gamma_{\II}^{(t)}\big)\right)\gs \Vol_d\left(\Delta\big(\Gamma_{\II_a}^{(t)}\big)\right)=\displaystyle\lim_{n\to \infty}\frac{\#\left[ \Gamma_{\II_{a,na}}^{(t)}\right] }{(na)^d}\gs \displaystyle\lim_{n\to \infty}\frac{\#\left[n\star \Gamma_{\II,{a}}^{(t)}\right]}{(na)^d}\gs \Vol_d\left( \Delta\big(\Gamma_{\II}^{(t)}\big)\right)-\varepsilon
\end{equation}
for every $1\ls t\ls [k_S:k]$. Then the second statement  follows from  
\autoref{cotaJ} and \autoref{cotaI}.
\end{proof}

We include the following observation for future reference.

\begin{remark}\label{remLim}
Assume  the notations introduced in this section.  
Proceeding similarly to \autoref{theLL}, from \cite[Corollary 1.16]{KAVEH_KHOVANSKII}  we obtain the following equality
$$\lim_{n\to \infty}\frac{\lambda\left( J_n/K_{\alpha c n}\cap J_n \right)}{n^d}=\sum_{t=1}^{[k_S:k]}  \Vol_d\left(\Delta\big(\Gamma_{\JJ}^{(t)}\big)\right). $$

\end{remark}

Our next goal is to prove \autoref{thm_main}, which is the main tool for our results in the next section. For this theorem we need to  introduce some prior notation.

\begin{setup}\label{setup_gen} We adopt \autoref{setup_lin_gr}. 
Let $\JJ(1)=\{J(1)_n\}_{n\in \NN}$,   $\ldots$, $\JJ(r)=\{J(r)_n\}_{n\in \NN}$ be graded families of non-zero ideals, and let $\II(1)=\{I(1)_n\}_{n\in \NN}$,   $\ldots$, $\II(s)=\{I(s)_n\}_{n\in \NN}$ be $\mm$-primary graded families of  ideals. 
 For $\bn=(n_1,\ldots, n_r)\in \NN^{r}$, $\bm=(m_1,\ldots, m_s)\in \NN^{s}$, and $p,a \in \NN$ we use the following notation:
  $$\bJ_\bn:=J(1)_{n_1}\cdots J(r)_{n_r}, \,\,\,\, \bI_\bm := I(1)_{m_1}\cdots I(s)_{m_s},$$
 \begin{equation}\label{powerNot} 
\bJ(p)^\bn:=J(1)_p^{n_1}\cdots J(r)_p^{n_r},\,\,\,\, 
 \bJ_{a}(p)^\bn:=J(1)_{a,p}^{n_1}\cdots J(r)_{a,p}^{n_r},\,\,\,\, 
  \bJ_{a,\bn}:=J(1)_{a,n_1}\cdots J(r)_{a,n_r},
\end{equation}
$$  
\bI(p)^\bm:=I(1)_p^{m_1}\cdots I(s)_p^{m_s},\,\,\,\, 
 \bI_{a}(p)^\bm:=I(1)_{a,p}^{m_1}\cdots I(r)_{a,p}^{m_s},\,\,\,\, 
  \bI_{a,\bm}:=I(1)_{a,m_1}\cdots I(r)_{a,m_s}.
  $$
 For  $\bm=(m_1,\ldots, m_s)\in \NN^{s}$ and $\bn=(n_1,\ldots, n_r)\in \NN^{r}$ we define the   pair of graded families 
  $$
  (\mathcal{J}_{\bm,\bn}, \mathcal{H}_{\bm,\bn}) \,:=\Big(\{\bJ_{m\bn}\}_{m\in \NN},\, \{\bI_{m\bm}\bJ_{m\bn}\}_{m\in \NN}\Big) \quad \text{and}
  $$ 
 $$
  (\mathcal{J}(p)_{\bm,\bn}, \mathcal{H}(p)_{\bm,\bn})\,:=\Big(\{\bJ(p)^{m\bn}\}_{m\in \NN},\, \{\bI(p)^{m\bm}\bJ(p)^{m\bn}\}_{m\in \NN}\Big)\,\, \text{for every } p\in \ZZ_{>0}.
  $$ 
 We further assume that  each $(\mathcal{J}_{\bm,\bn}, \mathcal{H}_{\bm,\bn})$
  has linear growth, and that if $c_{\bm,\bn}:=c (\mathcal{J}_{\bm,\bn}, \mathcal{H}_{\bm,\bn})$, then 
 $ c (\mathcal{J}(p)_{\bm,\bn}, \mathcal{H}(p)_{\bm,\bn}) = c_{\bm,\bn}p$ for every $p$.
\end{setup}


We continue with the following lemma which is needed in the proof of  \autoref{thm_main}.
It provides a further needed approximation result for the case of Noetherian filtrations, and it is the natural extension of \cite[Proposition 4.3]{MIXED_VOL_MONOM}.

\begin{lemma}\label{thm_p}
Assume \autoref{setup_gen}. Moreover, assume that $\JJ(i)$ and $\II(j)$ are all Noetherian for $1 \le i \le r$  and $1 \le j \le s$.  
For   every $\bm=(m_1,\ldots, m_s)\in \NN^{s}$ and $\bn=(n_1,\ldots, n_r)\in \NN^{r}$ we have 
$$
\lim_{p\to \infty}\Vol_d\left(\Delta\left(\Gamma_{\mathcal{J}(p)_{\bm,\bn}}^{(t)}\right)\right) = \Vol_d\left(\Delta\left(\Gamma_{\mathcal{J}_{\bm,\bn}}^{(t)}\right)\right)
\quad \text{ and }$$ 
$$\lim_{p\to \infty}\Vol_d\left(\Delta\left(\Gamma_{\mathcal{H}(p)_{\bm,\bn}}^{(t)}\right)\right) = \Vol_d\left(\Delta\left(\Gamma_{\mathcal{H}_{\bm,\bn}}^{(t)}\right)\right)
$$
for all $t \in \ZZ_{>0}$.
\end{lemma} 
\begin{proof}
We fix $\bm=(m_1,\ldots, m_s)\in \NN^{s}$ and $\bn=(n_1,\ldots, n_r)\in \NN^{r}$ and  consider the following graded families:

 $$
  (\mathcal{A}_{\bm,\bn}, \mathcal{B}_{\bm,\bn})\,:=\Big(\{\bJ(p)^{\bn}\}_{p\in \NN},\, \{\bI(p)^{\bm}\bJ(p)^{\bn}\}_{p\in \NN}\Big).
  $$ 
We note that $ c(\mathcal{A}_{\bm,\bn}, \mathcal{B}_{\bm,\bn})=c_{\bm,\bn}$,	then
	$$
	m \star \big[	\Gamma_{\mathcal{A}_{\bm,\bn}}^{(t)}\big]_{p} \subset \left[\Gamma_{\mathcal{J}(p)_{\bm,\bn}}^{(t)}\right]_{mp} \subset  \big[\Gamma_{\mathcal{A}_{\bm,\bn}}^{(t)}\big]_{mp}
	\;\; \text{ and } \;\; m \star \big[	\Gamma_{\mathcal{B}_{\bm,\bn}}^{(t)}\big]_{p} \subset \left[\Gamma_{\mathcal{H}(p)_{\bm,\bn}}^{(t)}\right]_{mp} \subset  \big[\Gamma_{\mathcal{B}_{\bm,\bn}}^{(t)}\big]_{mp}.
	$$
	for every $m,p,t\in \ZZ_{>0}$. Thus,  by \cite[Proposition 3.1]{lazarsfeld09} (see also \cite[Theorem 3.3]{cutkosky2014}) and \cite[Corollary 1.16]{KAVEH_KHOVANSKII}, for a fixed $\varepsilon\in \RR_{>0}$,  there exists $p_0\in \NN$ such that if $p \gs p_0$ we have
	\begin{equation}\label{eqp1}
		\Vol_{d}\left(\Delta\left(\Gamma_{\mathcal{A}_{\bm,\bn}}^{(t)}\right)\right) 
		\ge  
		\lim_{m\to \infty} \frac{\#\left[\Gamma_{\mathcal{J}(p)_{\bm,\bn}}^{(t)}\right]_{mp} }{m^dp^d} 
		\ge 
		\lim_{m\to \infty} \frac{\#\left[m \star \big[	\Gamma_{\mathcal{A}_{\bm,\bn}}^{(t)
		}\big]_p\right]}{m^dp^d} 
		\ge 
		\Vol_{d}\left(	\Delta\left(\Gamma_{\mathcal{A}_{\bm,\bn}}^{(t)}\right)\right)    - \varepsilon
	\end{equation}
	and
	\begin{equation}\label{eqp2}
		\Vol_{d}\left(\Delta\left(\Gamma_{\mathcal{B}_{\bm,\bn}}^{(t)}\right)\right) 
		\ge  
		\lim_{m\to \infty} \frac{\#\left[\Gamma_{\mathcal{H}(p)_{\bm,\bn}}^{(t)}\right]_{mp} }{m^dp^d} 
		\ge 
		\lim_{m\to \infty} \frac{\#\left[m \star \big[	\Gamma_{\mathcal{B}_{\bm,\bn}}^{(t)
		}\big]_p\right]}{m^dp^d} 
		\ge 
		\Vol_{d}\left(	\Delta\left(\Gamma_{\mathcal{B}_{\bm,\bn}}^{(t)}\right)\right)    - \varepsilon
	\end{equation}
	for every $ p,t\in \ZZ_{>0}$. Then, from \autoref{eqp1}, \autoref{eqp2}, and  \cite[Corollary 1.16]{KAVEH_KHOVANSKII} it follows that 

\begin{equation}\label{lim1}
\lim_{p\to \infty}\Vol_d\left(\Delta\left(\Gamma_{\mathcal{J}(p)_{\bm,\bn}}^{(t)}\right)\right) = \Vol_d\left(\Delta\left(\Gamma_{\mathcal{A}_{\bm,\bn}}^{(t)}\right)\right)
\quad \text{ and }
\end{equation} 

\begin{equation}\label{lim2}
\lim_{p\to \infty}\Vol_d\left(\Delta\left(\Gamma_{\mathcal{H}(p)_{\bm,\bn}}^{(t)}\right)\right) = \Vol_d\left(\Delta\left(\Gamma_{\mathcal{B}_{\bm,\bn}}^{(t)}\right)\right).
\end{equation} 

By the Noetherian assumption, there exists $q > 0$ such that 
	\begin{equation}
		\label{eq_Noetherian}
			J(i)_{q}^n=J(i)_{nq} \;\; \text{ and } \;\; I(j)_{q}^n=I(j)_{nq} \;\; \text{ for every } n\gs 0, 1\ls i\ls r, 1\ls j\ls s
	\end{equation}
	(see, e.g., \cite[Lemma 13.10]{GORTZ_WEDHORN}, \cite[Theorem 2.1]{herzog2007}). 
Therefore, $$\bJ(mq)^{\bn} = \bJ_{mq\bn}\quad\text{ and }\quad\bI(mq)^{\bm}\bJ(mq)^{\bn} = \bI_{mq\fm}\bJ_{mq\bn}$$ for all $m \ge 0$, and then   \cite[Corollary 1.16]{KAVEH_KHOVANSKII}  implies
	\begin{equation}\label{lim3}
	\Vol_d\left(\Delta\left(\Gamma_{\mathcal{J}_{\bm,\bn}}^{(t)}\right)\right)
	=
	\lim_{m\to \infty} \frac{\#\left[\Gamma_{\mathcal{J}_{\bm,\bn}}^{(t)}\right]_{mq} }{m^dq^d} 
	=
	\lim_{m\to \infty} \frac{\#\left[\Gamma_{\mathcal{A}_{\bm,\bn}}^{(t)}\right]_{mq} }{m^dq^d} 
	=  
		\Vol_{d}\left(\Delta\left(\Gamma_{\mathcal{A}_{\bm,\bn}}^{(t)}\right)\right)
	\end{equation}
	and
	\begin{equation}\label{lim4}
	\Vol_d\left(\Delta\left(\Gamma_{\mathcal{H}_{\bm,\bn}}^{(t)}\right)\right)
	=
	\lim_{m\to \infty} \frac{\#\left[\Gamma_{\mathcal{H}_{\bm,\bn}}^{(t)}\right]_{mq} }{m^dq^d} 
	=
	\lim_{m\to \infty} \frac{\#\left[\Gamma_{\mathcal{B}_{\bm,\bn}}^{(t)}\right]_{mq} }{m^dq^d} 
	=  
		\Vol_{d}\left(\Delta\left(\Gamma_{\mathcal{B}_{\bm,\bn}}^{(t)}\right)\right).
	\end{equation}
The result now follows by combining \autoref{lim1}, \autoref{lim2}, \autoref{lim3}, and  \autoref{lim4}.
\end{proof}

We are now ready to present the main theorem of this section.

\begin{theorem}
	\label{thm_main}
	Assume \autoref{setup_gen}.
For $\bn=(n_1,\ldots, n_r)\in \NN^{r}$ and $\bm=(m_1,\ldots, m_s)\in \NN^{s}$  we have that the following limits exist and are equal
	$$
\lim_{p \to\infty}\lim_{m\to \infty}\frac{\lambda\left( \bJ(p)^{m\bn}/\bI(p)^{m\bm}\bJ(p)^{m\bn}\right)}{p^dm^d} = \lim_{m\to \infty}\frac{\lambda\left( \bJ_{m\bn}/\bI_{m\bm}\bJ_{m\bn}\right)}{m^d}.
	$$
\end{theorem}
\begin{proof}
Fix $a\in \ZZ_{>0}$, and notice that $I(i)_{a,n} \subset I(i)_n$ and $J(j)_{a,n} \subset J(j)_n$ for all $1 \le i \le s$, $1 \le j \le r$ and $n \in \NN$. We fix $\bm=(m_1,\ldots, m_s)\in \NN^{s}$ and $\bn=(n_1,\ldots, n_r)\in \NN^{r}$, and 
for simplicity of notation, we set
$$(\mathcal{J}, \mathcal{H}) \,:=\,  (\mathcal{J}_{\bm,\bn}, \mathcal{H}_{\bm,\bn}),\quad  (\mathcal{J}(p), \mathcal{H}(p)) \,:=\,  (\mathcal{J}(p)_{\bm,\bn}, \mathcal{H}(p)_{\bm,\bn}),
\quad \text{and}\quad  c:=c(\mathcal{J}, \mathcal{H}).$$
With this notation, by \autoref{thm_limits} it suffices to show
$$\lim_{p\to\infty}\displaystyle\sum_{t=1}^{[k_S:k]}\left(\Vol_d\left(\Delta\big(\Gamma_{\mathcal{J}(p)}^{(t)}\big)\right)-\Vol_d\left(\Delta\big(\Gamma_{\mathcal{H}(p)}^{(t)}\big)\right)\right)=\displaystyle\sum_{t=1}^{[k_S:k]}\left(\Vol_d\left(\Delta\big(\Gamma_{\mathcal{J}}^{(t)}\big)\right)-\Vol_d\left(\Delta\big(\Gamma_{\mathcal{H}}^{(t)}\big)\right)\right).$$
Let $a':=a \cdot \max\{n_1,\ldots,n_r,m_1,\ldots,m_s\}$. 
Since $\mathcal{J}_a$ is generated by the ideals $\bJ_{\bn},\bJ_{2\bn},\ldots,\bJ_{a\bn}$, we get the inclusions $\mathcal{J}_{a,m} \subseteq \bJ_{a',m\bn} \subseteq \bJ_{m\bn} = \mathcal{J}_m$ for all $m \in \NN$.
 So, 
 by \autoref{remLim} we obtain the following
\begin{align}
	\label{eq_lim_ineqs_A}
	\begin{split}	
	\sum_{t=1}^{[k_S:k]}\Vol_d\left(\Delta\big(\Gamma_{\mathcal{J}_a}^{(t)}\big)\right) 
	= 
	\lim_{m\to \infty} \frac{\lambda\left(
		\frac{\mathcal{J}_{a,m}}{K_{\alpha c m}\,\cap\, \mathcal{J}_{a,m}}
		\right)}{m^d}
	\le  
	\lim_{m\to \infty} \frac{\lambda\left(
		\frac{\bJ_{a',m\bn}}{K_{\alpha c m}\,\cap\, \bJ_{a',m\bn}}
		\right)}{m^d} \le \\
 	\le
	\lim_{m\to \infty} \frac{\lambda\left(
		\frac{\mathcal{J}_{m}}{K_{\alpha c m}\,\cap\, \mathcal{J}_{m}}
		\right)}{m^d} 
	= \sum_{t=1}^{[k_S:k]}\Vol_d\left(\Delta\big(\Gamma_{\mathcal{J}}^{(t)}\big)\right). 
\end{split}
\end{align}
Likewise, we also have $\mathcal{H}_{a,m} \subseteq \bI_{a',m\bm}\bJ_{a',m\bn} \subseteq \bI_{m\bm}\bJ_{m\bn} = \mathcal{H}_{m}$ for all $m \in \NN$ and then
\begin{align}
		\label{eq_lim_ineqs_B}
	\begin{split}	
		\sum_{t=1}^{[k_S:k]}\Vol_d\left(\Delta\big(\Gamma_{\mathcal{H}_a}^{(t)}\big)\right) 
		= 
		\lim_{m\to \infty} \frac{\lambda\left(
			\frac{\mathcal{H}_{a,m}}{K_{\alpha c_ n}\,\cap\, \mathcal{H}_{a,m}}
			\right)}{m^d}
		\le  
		\lim_{m\to \infty} \frac{\lambda\left(
			\frac{\bI_{a',m\bm}\bJ_{a',m\bn}}{K_{\alpha c n}\,\cap\, \bI_{a',m\bm}\bJ_{a',m\bn}}
			\right)}{m^d} \le \\
		\le
		\lim_{m\to \infty} \frac{\lambda\left(
			\frac{\mathcal{H}_{m}}{K_{\alpha c m}\,\cap\, \mathcal{H}_{m}}
			\right)}{m^d} 
		= \sum_{t=1}^{[k_S:k]}\Vol_d\left(\Delta\big(\Gamma_{\mathcal{H}}^{(t)}\big)\right). 
	\end{split}
\end{align}
Therefore, by  \autoref{eq_lim_ineqs_A} and \autoref{eq_lim_ineqs_B} and by applying the second statement of \autoref{thm_limits}, for a given $\varepsilon>0$ there exists $a:=a(\varepsilon)\in \ZZ_{>0}$ such that 
\begin{equation}\label{con1}
\sum_{t=1}^{[k_S:k]}\Vol_d\left(\Delta\big(\Gamma_{\mathcal{J}}^{(t)}\big)\right)\gs \lim_{m\to \infty} \frac{\lambda\left(
		\frac{\bJ_{a',m\bn}}{K_{\alpha c m}\,\cap\, \bJ_{a',m\bn}}
		\right)}{m^d} \gs 
		\sum_{t=1}^{[k_S:k]}\Vol_d\left(\Delta\big(\Gamma_{\mathcal{J}}^{(t)}\big)\right) -\frac{\varepsilon}{2}, 
\end{equation}
and, 
\begin{equation}\label{con2}
\sum_{t=1}^{[k_S:k]}\Vol_d\left(\Delta\big(\Gamma_{\mathcal{H}}^{(t)}\big)\right)\gs \lim_{m\to \infty} \frac{\lambda\left(
			\frac{\bI_{a',m\bm}\bJ_{a',m\bn}}{K_{\alpha c n}\,\cap\, \bI_{a',m\bm}\bJ_{a',m\bn}}
			\right)}{m^d}  \gs 
		\sum_{t=1}^{[k_S:k]}\Vol_d\left(\Delta\big(\Gamma_{\mathcal{H}}^{(t)}\big)\right) -\frac{\varepsilon}{2}, 
\end{equation}
For every $p\in \ZZ_{>0}$ we have the inclusions
$$
 \bJ_{a'}(p)^{m\bn} \subseteq \mathcal{J}(p)_m  \subseteq \mathcal{J}_{pm} \qquad
\text{ and } \qquad 
 \bI_{a'}(p)^{m\bm}\bJ_{a'}(p)^{m\bn} \subseteq \mathcal{H}(p)_m\subseteq \mathcal{H}_{pm}.
$$
Thus, as each $\II(i)_{a'}$ and each $\JJ(j)_{a'}$ is a Noetherian graded family, by \autoref{thm_p} and \autoref{remLim} there exists $p_0:=p_0(a)$ such that if  $p\gs  p_0$  we have 
\begin{align}\label{con3}
	\begin{split}	
\sum_{t=1}^{[k_S:k]}\Vol_d\left(\Delta\big(\Gamma_{\mathcal{J}}^{(t)}\big)\right)\gs
\sum_{t=1}^{[k_S:k]}\Vol_d\left(\Delta\big(\Gamma_{\mathcal{J}(p)}^{(t)}\big)\right) \gs
\lim_{m\to \infty} \frac{\lambda\left(
		\frac{ \bJ_{a'}(p)^{m\bn}}{K_{\alpha c m}\,\cap\,  \bJ_{a'}(p)^{m\bn}}
		\right)}{m^d}  
\gs \\
\gs
\lim_{m\to \infty} \frac{\lambda\left(
		\frac{\bJ_{a',m\bn}}{K_{\alpha c m}\,\cap\, \bJ_{a',m\bn}}
		\right)}{m^d}   -\frac{\varepsilon}{2},
				\end{split}	
\end{align}
and
\begin{align}\label{con4}
	\begin{split}	
\sum_{t=1}^{[k_S:k]}\Vol_d\left(\Delta\big(\Gamma_{\mathcal{H}}^{(t)}\big)\right)\gs
\sum_{t=1}^{[k_S:k]}\Vol_d\left(\Delta\big(\Gamma_{\mathcal{H}(p)}^{(t)}\big)\right) 
\gs\lim_{m\to \infty} \frac{\lambda\left(
			\frac{ \bI_{a'}(p)^{m\bm}\bJ_{a'}(p)^{m\bn}}{K_{\alpha c n}\,\cap\, \bI_{a'}(p)^{m\bm}\bJ_{a'}(p)^{m\bn}}
			\right)}{m^d} 
\gs\\
\gs 
 \lim_{m\to \infty} \frac{\lambda\left(
			\frac{\bI_{a',m\bm}\bJ_{a',m\bn}}{K_{\alpha c n}\,\cap\, \bI_{a',m\bm}\bJ_{a',m\bn}}
			\right)}{m^d}   -\frac{\varepsilon}{2}.
				\end{split}	
\end{align}
The result now follows by combining \autoref{con1}, \autoref{con2}, \autoref{con3}, and \autoref{con4}.
\end{proof}

\section{Existence of mixed multiplicities of graded families}
\label{sect_exists}

In this  section, we use \autoref{thm_main} to show the existence of mixed multiplicities of graded families of ideals. 
We begin with the $\mm$-primary case.

\subsection{The $\mm$-primary case}

In this subsection, we use the following setup.

\begin{setup}\label{setup_mprim}
Let $(R,\mm,k)$ be a Noetherian local ring of dimension $d$ such that $\dim\left( N(\hat{R})\right)<d$; here $N(\hat{R})$ denotes the nilradical of the $\mm$-adic completion  $\hat{R}$. 
Let $M$ be a finitely generated $R$-module. 
Let  $\II(1)=\{I(1)_n\}_{n\in \NN}$,   $\ldots$, $\II(s)=\{I(s)_n\}_{n\in \NN}$  be $\mm$-primary graded families of ideals. 
 For  every $p\in \NN$ and $\bm=(m_1,\ldots, m_s)\in \NN^{s}$  we follow the same abbreviations from \autoref{powerNot}.
 The sequence $(\II(1),\ldots,\II(s))$ of graded families is simply denoted by $\II$.
For each $p\in \NN$, we denote by $G_{\II(p)}^M(t_1,\ldots, t_s)$  the polynomial 
$$
G_{(I(1)_p,\ldots, I(s)_p)}^M(t_1,\ldots, t_s)
$$ 
corresponding with the ideals $I(1)_p,\ldots,I(s)_p$ (see \autoref{polMMnot}).
\end{setup}

\begin{remark}\label{theMpLG}
We note that there exists $c\in \NN$ such that $\mm^c\subset I(i)_1$ for every $1\ls i\ls s$. Thus, for every  $\bm=(m_1,\ldots, m_s)\in \NN^{s}$, the assumptions of \autoref{setup_gen} are satisfied  if $J(i)_n=R$ for each $1\ls i\ls r$ and $n\in \NN$. 
\end{remark}

The following  result allows us to define the  mixed multiplicities of $\mm$-primary graded families of  ideals.

\begin{theorem}\label{thm_exists_m-prim}
Assume \autoref{setup_mprim}. 
Then, there exists a homogeneous polynomial of total degree $d$ and real coefficients $G_\II^M(\bt)=G_{(\II(1),\ldots,\II(s))}^M(t_1,\ldots, t_s)\in \RR[\ft]=\RR[t_1,\ldots, t_s]$ such that 
$$
G_\II^M(\bm) = \lim_{m\to \infty}\frac{\lambda\left( M/\bI_{m\bm}M\right)}{m^d}
$$
for every $\bm=(m_1,\ldots, m_s)\in \NN^s$.
Moreover, for every $\bm\in \NN^s$ we have 
$$\lim_{p\to\infty} \frac{G_{\II(p)}^M(\bm)}{p^d}=G_\II^M(\bm).$$
\end{theorem}

\begin{proof}
By passing to the $\mm$-adic completion $\hat{R}$, we may assume $R$ is a complete local ring. By assumption, and by following the same proof of \cite[Lemma 5.2]{cutkosky2019} we may assume $R$ is complete and reduced. Let $\fp_1,\ldots, \fp_u$ be the minimal primes of $R$ and set $R_i:=R/\fp_i$  for every  $1\ls i\ls u$.  In \cite[Lemma 5.4]{cutkosky2014} the authors prove that for any $\mm$-primary graded family of  ideals $\{I_m\}_{m\in \NN}$ that is a {\it filtration}, i.e.,  $I_{m+1}\subseteq I_m$ for every $m\in \NN$, we have 
\begin{equation}
	\label{eq_reduction_min_primes}
	\lim_{m\to\infty} \frac{\lambda(M/I_mM)}{m^d}=\sum_{i=1}^u\lim_{m\to\infty} \frac{\lambda(M\otimes_R R_i/I_mM\otimes_R R_i)}{m^d}.
\end{equation}
However, it is easy to see that the filtration condition is not necessary and that the result in \autoref{eq_reduction_min_primes} is also valid for graded families.
Thus, by using this fact first with $I_m=\bI_{m\fm}$ and then with $I_m=\bI(p)^{m\fm}$ we may assume $R$ is a complete domain. Under the latter assumption, in \cite[Lemma 5.3]{cutkosky2014} it is shown that 
\begin{equation}
	\label{eq_reduction_rank}
	\lim_{m\to\infty} \frac{\lambda(M/I_mM)}{m^d}=\rank_R(M)\left(\lim_{m\to\infty} \frac{\lambda(R/I_m)}{m^d}\right)
\end{equation}
for any $\mm$-primary filtration  $\{I_m\}_{m\in \NN}$. 
Again, the result in \autoref{eq_reduction_rank} is also valid for graded families.
Hence, it suffices to show the result for $M=R$ when $R$ is a complete domain. The result now follows by using   \autoref{thm_main}  with $J(i)_n=R$ for each $1\ls i\ls r$ and $n\in \NN$ (see also \autoref{theMpLG}, \cite[Lemma 3.2, proof of Theorem 4.5]{cutkosky2019}).
\end{proof}

We are now ready to define the mixed multiplicities of $\mm$-primary graded families of  ideals. 

\begin{definition}\label{defMMm-prim}
Let $G_\II^M(\ft)$ be the polynomial in the conclusion of \autoref{thm_exists_m-prim}. Write
\begin{equation*}\label{pol_grad_m-prim}
G_{\II}^M(\ft)=\sum_{|\fd|=d}\frac{ e_{\fd}(M; \II(1),\ldots, \II(s))}{d_1!\cdots d_s!} t_1^{d_1}\cdots t_s^{d_s}.
\end{equation*}
We define the real number number $e_{\fd}(M; \II(1),\ldots, \II(s))$ to be the {\it mixed multiplicity of $M$ of type $\fd$ with respect to $\II(1),\ldots, \II(s)$.}
\end{definition}

As an immediate consequence of \autoref{thm_exists_m-prim} we obtain the following ``Volume = Multiplicity formula" for mixed multiplicities of $\mm$-primary  graded families of ideals. 

\begin{corollary}\label{cor_vol=mult}
For every $\fd\in \NN^s$ with $|\fd|=d$ we have 
$$\lim_{p\to \infty} \frac{e_{\fd}(M; I(1)_p,\ldots, I(s)_p)}{p^d}=e_{\fd}(M; \II(1),\ldots, \II(s)).$$
In particular, the coefficients of $G_\II^M(\ttt)$ are non-negative, that is, $e_{\fd}(M; \II(1),\ldots, \II(s))\gs 0$ for every $\fd \in \NN^s$ with $|\dd| =d$.
\end{corollary}

In the case that the graded families $\II(1),\ldots, \II(s)$ are all the same, we obtain the following corollary.

\begin{corollary}\label{cor_one_filt}
Assume the graded families $\II(1),\ldots, \II(s)$ are all equal to $\LL=\{L_n\}_{n\in \NN}$. 
Then for every $\fd\in \NN^s$ with $|\fd|=d$ we have 
$e_{\fd}(M; \II(1),\ldots, \II(s))=e_d(M;\LL).$
\end{corollary}
\begin{proof}
It is easy to verify that for an $\mm$-primary ideal $I$ one has $e_{\fd}(M;I,\ldots, I)=e_d(M;I)$ for every $\fd\in \NN^s$ with $|\fd|=d$. The result now follows from \autoref{cor_vol=mult}.
\end{proof}

\subsection{The general case}
The data below is set in place during this subsection.

\begin{setup}\label{setup_gen_ex}
Let $(R,\mm,k)$ be an analytically irreducible ring of dimension $d$. 
Let $\JJ(1)=\{J(1)_n\}_{n\in \NN}$,   $\ldots$, $\JJ(r)=\{J(r)_n\}_{n\in \NN}$ be graded families of non-zero ideals, and let $\II=\{I_n\}_{n\in \NN}$ be an  $\mm$-primary graded family of ideals. 
 For  every $p\in \NN$ and $\bn=(n_1,\ldots, n_r)\in \NN^{r}$  we follow the same abbreviations from \autoref{powerNot}.  
 The sequence $(\JJ(1),\ldots,\JJ(r))$ of graded families is simply denoted by $\JJ$.
For each natural number $p\in \NN$, we denote by $G_{(\II(p);\JJ(p))}(t_0,t_1,\ldots, t_r)$  the polynomial 
$$
G_{(I_p;J(1)_p,\ldots, J(r)_p)}^M(t_0,t_1,\ldots, t_r)
$$ 
corresponding with the ideals $I_p,J(1)_p,\ldots,J(r)_p$ (see \autoref{polGenMMnot}).
We further assume that for every natural number $n_0\in \NN$ and $\bn=(n_1,\ldots, n_r)\in \NN^{r}$ the pair of graded families 
  $$
  \Big(\{\bJ_{m\bn}\}_{m\in \NN},\, \{I_{mn_0}\bJ_{m\bn}\}_{m\in \NN}\Big)
  $$ 
  has linear growth, and that if $c:=c\left(\{\bJ_{m\bn}\}_{m\in \NN}, \{I_{mn_0}\bJ_{m\bn}\}_{m\in \NN}\right)$, then 
  $$\bJ(p)^{m\bn}\cap m^{cpm}=I_p^{mn_0}\bJ(p)^{m\bn}\cap m^{cpm}  \mbox{ for every }p,m\in \NN.$$
\end{setup}

\begin{remark}\label{remLinGr}
We note that the  assumptions in \autoref{setup_gen_ex} are natural  in the context of this paper (see, e.g., \cite[Theorem 6.1]{cutkosky2014}), and they are  satisfied under mild assumptions on the ideals. 
For instance, if we consider a positively graded ring over a field  and assume that there is a linear bound on $n$ for the degrees of the generators of the (now) homogeneous ideals $J(i)_n$ (cf. \cite[Lemma 3.9]{MIXED_VOL_MONOM}). The latter bound exists whenever $R$ is a standard polynomial ring and there is a linear bound for the Castelnuovo-Mumford regularities, for example, if $\{J_n\}_{n\in \NN}$ are the symbolic powers of ideals of dimension at most two \cite[Corollary 2.4]{chandler1997}, or initial ideals of ideals of dimension at most one \cite[Theorem 3.5]{herzog2002}.
\end{remark}

From the previous remark we obtain the following explicit examples.

\begin{example}
Let $X_1, \ldots, X_r\subseteq \PP_k^d$ be zero- or one-dimensional schemes and $J_1,\ldots, J_r$ their corresponding defining ideals in $ k[x_0,\ldots, x_d]$. The families of symbolic powers $\JJ(i)=\{J_i^{(n)}\}_{n\in \NN}$ for $1\ls i\ls r$, together with any $\mm$-primary graded family of homogeneous ideals $\II=\{I_n\}_{n\in \NN}$, satisfy the assumptions of \autoref{setup_gen_ex}. 
\end{example}

\begin{example}
An example of particular interest is the following: given a sequence of convex bodies $K_1,\ldots,K_r$, we consider the graded families of monomial ideals $\JJ(1),\ldots, \JJ(r)$ from \cite[Section 5]{MIXED_VOL_MONOM} whose mixed multiplicities coincide with the mixed volumes of $K_1,\ldots,K_r$ (see \cite[Theorem 5.4]{MIXED_VOL_MONOM}). These families, together with any $\mm$-primary graded family of homogeneous ideals $\II=\{I_n\}_{n\in \NN}$, satisfy the assumptions of \autoref{setup_gen_ex}. 
\end{example}

The following result allows us to define the  mixed multiplicities of graded families of ideals.

\begin{theorem}\label{thm_exists_general}
Assume \autoref{setup_gen_ex}. 
Then, there exists a homogeneous polynomial of total degree $d$ and  real coefficients $G_{(\II;\JJ)}(t_0,\bt)=G_{(\II;\JJ)}(t_0,t_1,\ldots, t_r)\in \RR[\ft]=\RR[t_0,t_1,\ldots, t_s]$ such that 
$$
G_{(\II;\JJ)}(n_0,\bn) = \lim_{m\to \infty}\frac{\lambda\left( \bJ_{m\bn}/I_{mn_0}\bJ_{m\bn}\right)}{m^d},
$$
for every $n_0\in \NN$ and $\bn=(n_1,\ldots, n_r)\in \NN^r$. 
Additionally, the polynomial $G_{(\II;\JJ)}(t_0,\bt)$ has no term of the form $et_1^{d_1}\cdots t^{d_r}$ with $e\neq 0$ and $d_1+\cdots+d_r=d$. 
Moreover, for every $n_0$ and $\bn\in \NN^r$ we have 
$$\lim_{p\to\infty} \frac{G_{(\II(p);\JJ(p))}(n_0,\bn)}{p^d}=G_{(\II;\JJ)}(n_0,\bn).$$
\end{theorem}

\begin{proof}
By passing to the $\mm$-adic completion $\hat{R}$, we may assume $R$ is a complete local domain. 
The result now follows by  \autoref{thm_main} and \autoref{polGenMMnot} (see also \autoref{theMpLG}, \cite[Lemma 3.2, proof of Theorem 4.5]{cutkosky2019}).
\end{proof}

We are ready to define the mixed multiplicities of graded families of ideals. 

\begin{definition}\label{defMMm-general}
Let $G_{(\II;\JJ)}(t_0,\bt)$ be the polynomial in the conclusion of \autoref{thm_exists_general}. Write
\begin{equation}\label{pol_grad_general}
G_{(\II;\JJ)}(t_0,\bt)=\sum_{ d_0+|\fd|=d-1}\frac{ e_{(d_0,\fd)}(\II\mid \JJ(1),\ldots, \JJ(r))}{(d_0+1)!d_1!\cdots d_s!} t_0^{d_0+1}t_1^{d_1}\cdots t_r^{d_r}.
\end{equation}
We define the real number $e_{(d_0,\fd)}(\II\mid \JJ(1),\ldots, \JJ(r))$ to be the {\it mixed multiplicity of $\JJ(1),\ldots, \JJ(r)$ of type $(d_0,\fd)$ with respect to $\II$.}
\end{definition}

We also obtain  the following  version of the ``Volume = Multiplicity formula".

\begin{corollary}\label{cor_vol=mult_general}
For every $d_0\in \NN$ and $\fd\in \NN^r$ with $d_0+|\fd|=d-1$ we have 
$$\lim_{p\to \infty} \frac{e_{(d_0,\fd)}( I_p\mid J(1)_p,\ldots, J(r)_p)}{p^d}=e_{(d_0,\fd)}(\II\mid \JJ(1),\ldots, \JJ(r)).$$
In particular, the coefficients of $G_{(\II;\JJ)}(t_0,\bt)$ are non-negative, that is, $e_{(d_0,\fd)}(\II\mid  \JJ(1),\ldots, \JJ(r))\gs 0$ for every $d_0 \in \NN$ and $\fd \in \NN^r$ with $d_0+|\dd| = d-1$.
\end{corollary}

We end this section with the following comparison of the two notions of mixed multiplicities introduced in this section (cf. \cite[Theorem 1.2]{trung2007mixed}).

\begin{corollary}
Assume that $\JJ(1),\ldots,\JJ(r)$ are  $\mm$-primary graded families of  ideals. 
Then, for every $d_0\in \NN$ and $\fd\in \NN^r$ with $d_0+|\fd|=d$ we have 
\begin{enumerate}[\rm (i)]
		\item If $d_0 = 0$, then $e_{(d_0,\dd)}(R;\II,\JJ(1),\ldots,\JJ(r)) = e_{\dd}(R;\JJ(1),\ldots,\JJ(r))$.
		\item If $d_0 > 0$, then $e_{(d_0,\dd)}(R;\II,\JJ(1),\ldots,\JJ(r)) = e_{(d_0-1,\dd)}(\II\mid\JJ(1),\ldots,\JJ(r))$.
	\end{enumerate} 
	In particular, $e_{(d-1,\mathbf{0})}(\II\mid\JJ(1),\ldots,\JJ(r))=e_d(R;\II)$.
\end{corollary}
\begin{proof}
The result follows from the following short exact sequence 
	$$
	0 \rightarrow  \bJ_{m\bn}/I_{mn_0}\bJ_{m\bn} \rightarrow R/I_{mn_0}\bJ_{m\bn} \rightarrow R/\bJ_{m\bn} \rightarrow 0
	$$
	for every $n_0, m\in\NN$ and  $\bn\in \NN^r$.
\end{proof}

\section{Properties of mixed multiplicities of $\mm$-primary graded families}
\label{sect_properties}

In this short section we demonstrate how \autoref{thm_exists_m-prim} and \autoref{cor_vol=mult} can be used to show that  the mixed multiplicities of graded families  inherit many important properties from mixed multiplicities of ideals.  Throughout this section we assume \autoref{setup_mprim}.

We begin with the additivity under short exact sequences (cf. \cite[Lemma 17.4.4]{huneke2006integral}, \cite[Proposition 6.7]{cutkosky2019}).

\begin{proposition}\label{additivity}
Assume \autoref{setup_mprim}. Let 
$0\to M'\to M\to M''\to 0$
be a short exact sequence of finitely generated $R$-modules. Then for every $\fd\in \NN^s$ with $|\fd|=d$ we have 
$$e_{\fd}(M; \II(1),\ldots, \II(s))=e_{\fd}(M'; \II(1),\ldots, \II(s))+e_{\fd}(M''; \II(1),\ldots, \II(s)).$$
\end{proposition}
\begin{proof}
The result follows by \cite[Lemma 17.4.4]{huneke2006integral} and \autoref{cor_vol=mult}.
\end{proof}

We continue with the associativity formula (cf. \cite[Theorem 17.4.8]{huneke2006integral}, \cite[Theorem 6.8]{cutkosky2019}). 
In the following statement, for a graded family $\II=\{I_n\}_{n\in \NN}$, and a prime ideal $\fp\in \Spec(R)$, we denote by $\II(R/\fp)$ the graded family of $R/\fp$-ideals $\II(R/\fp)=\{I_n(R/p)\}_{n\in \NN}$.

\begin{theorem}\label{associativity}
Assume \autoref{setup_mprim}. Let $M$ be a finitely generated $R$-module. Then for every $\fd\in \NN^s$ with $|\fd|=d$ we have 
$$e_{\fd}(M; \II(1),\ldots, \II(s))=\sum_{\fp}\lambda_{R_\fp}(M_\fp)e_{\dd}\big(R/\fp;\II(1)(R/\fp),\ldots, \II(s)(R/\fp)\big),$$
where the sum runs over the minimal primes $\fp$ of  $R$ such that $\dim(R/\fp)=d$.
\end{theorem}
\begin{proof}
We note that for a minimal prime $\fp$ of $R$ we have $\dim\left(N(\widehat{R/\fp})\right)<d$ (see \cite[Theorem 6.8]{cutkosky2019}). The result now follows by \cite[Theorem 17.4.8]{huneke2006integral} and \autoref{cor_vol=mult}.
\end{proof}

We also obtain Minkowski inequalities for mixed multiplicities of graded families (cf. \cite[Theorem 17.7.2, Corollary 17.7.3]{huneke2006integral}, \cite[Theorem 6.3]{cutkosky2019}). In the following statement for a graded families $\II=\{I_n\}_{n\in \NN}$ and $\JJ=\{J_n\}_{n\in \NN}$ we denote by $\II\JJ$ the graded family $\II\JJ=\{I_nJ_n\}_{n\in \NN}$.

\begin{theorem}\label{Minkowski}
Assume \autoref{setup_mprim} and that $R$ has positive dimension. Then,
\begin{enumerate}[\rm (i)]
\item $e_{(i, d-i)}(M;\II(1), \II(2))^2\ls e_{(i+1, d-i-1)}(M;\II(1), \II(2))e_{(i-1, d-i+1)}(M;\II(1), \II(2)) $ for $1\ls i\ls d-1$.
\item $e_{(i, d-i)}(M;\II(1), \II(2))e_{(d-i, i)}(M;\II(1), \II(2))\ls e_{d}(M;\II(1))e_{d}(M;\II(2)) $ for $0\ls i\ls d$. 
\item $e_{(i, d-i)}(M;\II(1), \II(2))^d\ls e_{d}(M;\II(1))^{d-i}e_{d}(M;\II(2))^{i}$ for $0\ls i\ls d$. 
\item $e_{d}(M;\II(1)\II(2))^{\frac{1}{d}}
\ls e_{d}(M;\II(1))^{\frac{1}{d}}+e_{d}(M;\II(2))^{\frac{1}{d}}$ for $0\ls i\ls d$. 
\end{enumerate}
\end{theorem}
\begin{proof}
The result follows by \cite[Theorem 17.7.2, Corollary 17.7.3]{huneke2006integral} and \autoref{cor_vol=mult}.
\end{proof}







\section*{Acknowledgments}

The second  author is  supported by NSF Grant DMS \#2001645.

\bibliographystyle{elsarticle-num} 
\bibliography{references}

\end{document}